\documentclass{amsart}

\usepackage{geometry}
\usepackage{amsmath}
\usepackage{amssymb}
\usepackage{mathrsfs}
\usepackage{MnSymbol}
\usepackage{graphicx}
\usepackage{float}
\usepackage{amsthm}
\usepackage{tikz-cd}
\usepackage{cancel}
\usepackage{dynkin-diagrams}
\usepackage{hyperref}
\usepackage{tikz}
\usepackage[utf8]{inputenc}

\title{Algebraic hyperbolicity of adjoint linear systems on spherical varieties}
\author{Minseong Kwon}
\address{Morningside Center of Mathematics, Academy of Mathematics and Systems Science, Chinese Academy of Sciences, Beijing 100190, China}
\email{minseong@amss.ac.cn}

\author{Haesong Seo}
\address{Department of Mathematical Sciences, KAIST, 291 Daehak-ro, Yuseong-gu, Daejeon 34141, Republic of Korea, and Center for Complex Geometry, Institute for Basic Science (IBS), 55 Expo-ro, Yuseong-gu, Daejeon 34126, Republic of Korea}
\email{hss21@kaist.ac.kr}
\date{August 13, 2025, February 9, 2026 (v2)}

\subjclass[2020]{32Q45, 14J70, 14M27} 
\keywords{algebraic hyperbolicity, adjoint linear system, spherical variety}

\newtheorem{thm}{Theorem}[section]
\newtheorem{coro}[thm]{Corollary}
\newtheorem{lemma}[thm]{Lemma}
\newtheorem{prop}[thm]{Proposition}

\newtheorem{conj}[thm]{Conjecture}
\newtheorem{defnthm}[thm]{Definition-Theorem}

\theoremstyle{definition}
\newtheorem{defn}[thm]{Definition}
\newtheorem{notation}[thm]{Notation}
\newtheorem{construction}[thm]{Construction}

\theoremstyle{remark}
\newtheorem{rmk}[thm]{Remark}

\def\BigRoman{\uppercase\expandafter{\romannumeral\number\count 255}}
\def\Romannumeral{\afterassignment\BigRoman\count255=}
\DeclareMathAlphabet{\mathpzc}{OT1}{pzc}{m}{it}

\def \CC {\mathbb{C}}

\def \PP {\mathbb{P}}

\def \ZZ {\mathbb{Z}}

\def \Ocal {\mathcal{O}}

\def \Scal {\mathcal{S}}

\def \Ucal {\mathcal{U}}
\def \Vcal {\mathcal{V}}

\def \Ycal {\mathcal{Y}}

\def \hbar {\bar{h}}

\begin{document}
    
\begin{abstract}
    Moraga and Yeong conjectured that for a smooth complex projective variety $X$ of dimension $n$, an ample line bundle $A$ on $X$ and an integer $m \ge 3 n + 1$, very general elements of the adjoint linear system $|\omega_{X} \otimes A^{\otimes m}|$ are algebraically hyperbolic.
    We prove the conjecture for spherical varieties with smooth orbit closures.
    As a corollary, we conclude that the conjecture holds for horospherical varieties, and for toroidal spherical varieties.
    Furthermore, for any spherical variety, we show that the conjecture holds modulo the complement of an open dense orbit.
\end{abstract}

\maketitle

\section{Introduction}

We are working in the category of algebraic varieties over $k$, an uncountable and algebraically closed field of characteristic $0$.
Algebraic hyperbolicity is an algebro-geometric analogue of Kobayashi hyperbolicity, introduced by Demailly \cite{Demailly97}.
Roughly, for a smooth projective variety $S$, as Kobayashi hyperbolicity is equivalent to non-existence of entire curves on $S$, algebraic hyperbolicity requires a lower bound for genera of complete curves on $S$ so that $S$ contains no rational or elliptic curve
(see Definition \ref{defn: hyp} for the precise definition).

It is well known that very general hypersurfaces $S$ on a smooth projective variety $X$ are algebraically hyperbolic, if their degree is sufficiently large (see \cite{Demailly20} and the references therein).
In general, it is a challenging problem to determine an optimal lower bound for the degree that guarantees algebraic hyperbolicity. Recently, after the work of Hasse and Ilten \cite{HI2021} for the case where $X$ is a toric 3-fold, there has been a significant progress on algebraic hyperbolicity of hypersurfaces on almost homogeneous varieties. Namely, Coskun and Riedl \cite{CR2023} introduced a machinery for checking algebraic hyperbolicity of very general hypersurfaces on almost homogeneous varieties, based on the seminal works of Clemens \cite{Clemens1986}, Ein \cite{Ein1988, Ein1991}, Voisin \cite{Voisin1996, Voisin1998}, and Pacienza \cite{Pacienza2003}.
Using their technique, Yeong \cite{Yeong2022} considered the case where $X$ is the product of the projective spaces, Mioranci \cite{Mioranci2023} studied the case where $X$ is homogeneous, and the second author \cite{Seo2025} investigated the case where $X$ is a general Fano 3-fold of Picard number 1 (based on Mukai's vector bundle method).

Our project is motivated by the recent work of Moraga and Yeong \cite{MY2024} who considered adjoint linear systems on toric varieties. In fact, inspired by Fujita's conjecture on positivity of adjoint line bundles, they conjectured the following:

\begin{conj}[{\cite{MY2024}}] \label{conj: adjoint bundles}
    Let $X$ be a smooth projective variety of dimension $n \ge 2$, and $A$ an ample line bundle on $X$.
    For all $m \ge 3n+1$, very general elements of the adjoint linear system $|\omega_{X} \otimes A^{\otimes m}|$ are algebraically hyperbolic.
\end{conj}

Note that the lower bound $3n+1$ cannot be lowered by the case $(X,\, A) = (\PP^{2},\, \Ocal_{\PP^{2}}(1))$ (since in this case, $\omega_{X} \otimes A^{\otimes 3n} \simeq \Ocal_{\PP^{2}}(3)$).

As pointed out by Moraga and Yeong \cite[Example 5.2]{MY2024}, it is known that Conjecture \ref{conj: adjoint bundles} holds for some rational homogeneous spaces, by the results of Ein \cite{Ein1988} (for the projective spaces) and of Mioranci \cite{Mioranci2023} (for Grassmannians and their products).
Moreover, Moraga and Yeong \cite{MY2024} proved Conjecture \ref{conj: adjoint bundles} for smooth toric varieties.
It is worth observing that one can discuss Conjecture~\ref{conj: adjoint bundles} even when $X$ has mild singularities.
In fact, in \cite{MY2024}, Conjecture~\ref{conj: adjoint bundles} is verified for Gorenstein toric 3-folds, and for any Gorenstein toric varieties, Conjecture~\ref{conj: adjoint bundles} is proven {\lq modulo toric divisors\rq} (Definition~\ref{defn: hyp}).

We investigate Conjecture \ref{conj: adjoint bundles} for a larger class of almost homogeneous varieties, called spherical varieties.
Here, for a connected reductive group $G$, a normal $G$-variety $X$ is called \emph{spherical} if a Borel subgroup of $G$ has an open dense orbit in $X$.
For example, rational homogeneous spaces and toric varieties are spherical.
For more examples, we refer to \cite[Chapter 5]{Timashev}.

Our main result is the following theorem:

\begin{thm}\label{main thm: conj holds}
    Let $G$ be a connected reductive group, and $X$ a smooth projective spherical $G$-variety of dimension $n \ge 2$.
    If all $G$-orbit closures in $X$ are smooth, then Conjecture \ref{conj: adjoint bundles} holds for $X$.
\end{thm}

As a corollary, we verify Conjecture \ref{conj: adjoint bundles} for two important subclasses of spherical varieties.

\begin{coro} \label{coro: hyperbolic if horosph or regular}
    Conjecture \ref{conj: adjoint bundles} holds, provided that either
    \begin{enumerate}
        \item $X$ is \emph{horospherical}; or
        \item $X$ is spherical and \emph{toroidal}.
    \end{enumerate}
\end{coro}

For the definitions of horospherical varieties and toroidal spherical varieties, we refer to Definitions \ref{defn: horo} and \ref{defn: toroidal}, respectively.
Roughly, horospherical varieties are spherical varieties birational to torus bundles over rational homogeneous spaces, and Pasquier \cite[Proposition 2.17]{pasquier-thesis} showed that every $G$-orbit closure in a smooth horospherical variety is smooth.
On the other hand, toroidal spherical varieties are spherical varieties locally isomorphic to the products of toric varieties and affine spaces, and hence every $G$-orbit closure in a smooth toroidal spherical variety is smooth, see for example \cite[Corollary 3.4.3]{Perrin-sanya}.
In particular, we conclude that Conjecture \ref{conj: adjoint bundles} is true for
\begin{itemize}
    \item rational homogeneous spaces (including the examples of \cite[Example 5.2]{MY2024});
    \item smooth toric varieties (which were considered by Moraga and Yeong \cite{MY2024}); and
    \item de Concini--Procesi compactification of symmetric varieties \cite{dCP83}, or more generally, wonderful varieties \cite[\S30]{Timashev}.
\end{itemize}

To show Theorem \ref{main thm: conj holds}, indeed we prove a stronger statement, formulated as follows:

\begin{thm}\label{main thm: hyperbolic if orbit closures are smooth}
    Let $G$ be a connected reductive group, and $X$ a smooth projective spherical $G$-variety of dimension $n \ge 2$.
    Suppose that $N$ is a nef line bundle and $A$ is an ample line bundle on $X$.
    If all $G$-orbit closures in $X$ are smooth, then for any $m \ge 2n$, very general elements of the linear system $|N \otimes A^{\otimes m}|$ are algebraically hyperbolic.
\end{thm}

Observe that Theorem \ref{main thm: hyperbolic if orbit closures are smooth} implies Theorem \ref{main thm: conj holds}, since
\begin{equation} \label{eqn:decomp adjoint}
    \omega_{X} \otimes A^{\otimes m} \simeq \left( \omega_{X} \otimes A^{\otimes (n+1)} \right) \otimes A^{\otimes (m-(n+1))},
\end{equation}
and $\omega_{X} \otimes A^{\otimes (n+1)}$ is nef (\cite[Example 1.5.35]{LazarsfeldI}).
In fact, since $X$ is spherical, $\omega_{X} \otimes A^{\otimes (n+1)}$ is furthermore globally generated (cf. Theorem~\ref{thm:facts-on-spherical}(3)), as expected by Fujita's conjecture which predicts that the global generation holds true for any smooth projective variety.
We refer to \cite[\S10.4.A]{LazarsfeldII} for a detailed discussion on Fujita's conjecture.

To prove Theorem \ref{main thm: hyperbolic if orbit closures are smooth}, we use the technique of Coskun and Riedl \cite{CR2019, CR2023}.
From this point of view, one may observe that the nefness of the twisted kernel bundle (Definition \ref{defn:kernel}) of $N \otimes A^{\otimes m}$ is closely related to the desired statement.
To obtain its nefness, we prove the following two Theorems~\ref{thm: res for ample is surjective} and \ref{thm:spherical-Nakai-criterion}, which have their own interest and hold true even without the uncountability of the base field $k$ and the smoothness of the spherical variety $X$.
The first is an analogue of the results of Mehta and Ramanathan \cite[Theorem~3]{MR85} and of Brion and Inamdar \cite[p.~218]{BrionInamdar} (see also Theorem~\ref{thm:facts-on-spherical}(3)).
In the following, for an algebraic group $B$ acting on $X$, we say that a curve $C$ on $X$ is a \emph{$B$-curve} if $C$ is invariant under the $B$-action.

\begin{thm} \label{thm: res for ample is surjective}
    Assume that $k$ is an algebraically closed field of characteristic $0$.
    Let $G$ be a connected reductive group, and $B$ a Borel subgroup of $G$.
    Let $X$ be a projective spherical $G$-variety, and $A$ an ample line bundle on $X$.
    If $C$ is an irreducible complete $B$-curve on $X$, then the restriction map $H^{0}(X,\, A)\rightarrow H^{0}(C,\, A|_{C})$ is surjective.
\end{thm}

We obtain Theorem~\ref{thm: res for ample is surjective} as a consequence of Brion's classification of $B$-curves \cite{Brion1993, Brion97} and a characteristic $p$ method, the so-called Frobenius splitting, utilized as a main tool in the aforementioned results \cite{MR85, BrionInamdar}.

The second ingredient for the nefness of the twisted kernel bundle is a spherical version of the Nakai criterion for $B$-equivariant vector bundles.
We prove it by using the argument of Hering, Musta\c t\u a and Payne \cite[Theorem~2.1]{HMP2010} for the proof of the toric Nakai criterion.
The precise statement is as follows:

\begin{thm}\label{thm:spherical-Nakai-criterion}
    Assume that $k$ is an algebraically closed field of characteristic $0$.
    Let $G$ be a connected reductive group, and $B$ a Borel subgroup of $G$.
    Let $X$ be a complete spherical $G$-variety, and $E$ a $B$-equivariant vector bundle on $X$.
    \begin{enumerate}
        \item $E$ is nef on $B$-curves on $X$ if and only if $E$ is nef on $X$.

        \item Assume further that $X$ is projective.
        $E$ is ample on $B$-curves on $X$ if and only if $E$ is ample on $X$.
    \end{enumerate}
\end{thm}

In the process of the proof of Theorem \ref{main thm: hyperbolic if orbit closures are smooth}, we show that even in the case where the orbit closures are not necessarily smooth, very general elements of $|N \otimes A^{\otimes m}|$ are algebraically hyperbolic {\lq modulo the boundary.\rq}
As a consequence, Conjecture \ref{conj: adjoint bundles} holds {\lq modulo the boundary\rq} for any smooth spherical variety.
See Theorem \ref{main thm: pseudo hyperbolic} and Corollary \ref{coro: conj modulo boundary} for details.

This paper is organized as follows.
In Section \ref{sec:spherical-geometry}, we study a general theory of spherical varieties.
In Section \ref{subsection:pre}, we recall known results on spherical varieties which are used in sequel, especially Brion's classification of $B$-curves \cite{Brion1993, Brion97}.
In Section \ref{subsection: F splitting}, using the Frobenius splitting, we prove Theorem \ref{thm: res for ample is surjective}.
In Section \ref{subsection: Nakai}, we prove Theorem \ref{thm:spherical-Nakai-criterion}, and then show that the twisted kernel bundle is nef (Lemma \ref{lem:positivity-of-kernel-bundle}).
Section \ref{sec:hyperbolicity-conjecture} is devoted to the proof of our main theorem, Theorem \ref{main thm: hyperbolic if orbit closures are smooth}.
In Section \ref{subsection: coskun-riedl}, we summarize the technique of Coskun and Riedl \cite{CR2019, CR2023}.
Finally, in Section \ref{sec: Proof}, we complete the proof of Theorem \ref{main thm: hyperbolic if orbit closures are smooth}.

\subsection*{Conventions}
    We are working over $k$, an algebraically closed field of characteristic $p \ge 0$.
    We shall assume that $p = 0$ from Section \ref{subsection: Nakai}, and $k$ is furthermore uncountable in Section \ref{sec:hyperbolicity-conjecture}.
    By a variety, we mean an integral separated scheme of finite type over $k$.
    For the projectivization, we use Grothendieck's notation, that is, $\PP(V)$ parametrizes one-dimensional quotients of a vector space/bundle $V$.
    We say that a vector bundle $V$ is nef (resp. ample) if the line bundle $\Ocal_{\PP(V)}(1)$ on $\PP(V)$ is nef (resp. ample).

\subsection*{Acknowledgments}

The authors would like to thank Michel Brion, Izzet Coskun, Donggun Lee and Yongnam Lee for detailed comments on the draft of this paper.
The authors are grateful to Jaehyun Hong for introducing enlightening references.
The authors also would like to thank the anonymous referee for very helpful suggestions to improve this paper.
This work was supported by the Institute for Basic Science (IBS-R032-D1).

\section{Spherical geometry} \label{sec:spherical-geometry}

In this section, we discuss geometry of spherical varieties.
Namely, using Brion's classification of $B$-curves and the Frobenius splitting, we prove that the restriction map for global sections of an ample line bundle to a $B$-curve is surjective (Theorem \ref{thm: res for ample is surjective}).
Next, by Theorem \ref{thm: res for ample is surjective} and the spherical Nakai criterion (Theorem \ref{thm:spherical-Nakai-criterion}), we obtain the nefness of the twisted kernel bundles (Lemma \ref{lem:positivity-of-kernel-bundle}).

\subsection{Preliminaries} \label{subsection:pre}

Let us review well-known properties of spherical varieties.
Assume that $k$ is an algebraically closed field of arbitrary characteristic.

\begin{defn} \label{defn: spherical}
    Let $G$ be a connected reductive group, and $B$ a Borel subgroup.
    A normal $G$-variety $X$ is called \emph{spherical} if $B$ has an open dense orbit in $X$.
\end{defn}

\begin{thm}[{\cite[Theorem 2.2.12, Corollaries 3.3.5, 4.3.10]{Perrin-sanya}, \cite[Corollary 31.7]{Timashev}}] \label{thm:facts-on-spherical}
    Let $G$ be a connected reductive group, and $X$ a spherical $G$-variety.
    \begin{enumerate}
        \item For a Borel subgroup $B$ of $G$, $X$ contains only finitely many $B$-orbits.
        \item If $k$ is of characteristic $0$, then every $G$-invariant subvariety of $X$ is normal, and spherical with respect to the $G$-action.
        \item Assume that $k$ is of characteristic $0$ and $X$ is complete. For a line bundle $M$ on $X$, $M$ is nef if and only if $M$ is globally generated.
        If it is the case, then $H^{i}(X,\, M) = 0$ for all $i \ge 1$, and the restriction map $H^{0}(X,\, M) \rightarrow H^{0}(Y,\, M|_{Y})$ is surjective for any closed $G$-invariant subvariety $Y \subset X$.
    \end{enumerate}
\end{thm}

We also recall the notion of parabolic induction, which is a useful tool to construct a spherical variety.

\begin{defn} \label{defn:parabolic-ind}
    Let $G$ be a connected reductive group, and $P$ a parabolic subgroup of $G$.
    Suppose that $G_{0}$ is a reductive group acting on a variety $Y$, and that there is a surjective homomorphism $P \twoheadrightarrow G_{0}$.
    Considering $Y$ as a $P$-variety, the geometric quotient
    \[
        G\times^{P}Y:= (G\times Y)/ P
    \]
    with respect to the $P$-action on $G \times Y$ given by $p \cdot (g,\,y) := (g\cdot p^{-1},\, p \cdot y)$ exists as a $G$-variety (see Remark \ref{rmk: associated fiber bundle}(1)), and we call $G\times^{P}Y$ the \emph{parabolic induction} from $Y$ via $P \twoheadrightarrow G_{0}$.
\end{defn}

Given a parabolic induction $G \times^{P} Y$, we denote by $[g,\,y]$ the point represented by a pair $(g,\,y) \in G \times Y$.

\begin{rmk} \label{rmk: associated fiber bundle}
    \begin{enumerate}
        \item The parabolic induction $G \times^{P} Y$ is indeed a $G$-variety, equipped with a $G$-equivariant morphism
        \[
            G \times^{P} Y \rightarrow G/P, \quad [g,\,y] \mapsto g \cdot P
        \]
        whose fiber at $e \cdot P$ is $P \times^{P}Y \simeq Y$.
        It follows from the fact that $G \rightarrow G/P$ is a Zariski locally trivial fiber bundle, see \cite[\S 2.3]{Brion18} for details. (In fact, parabolic inductions are particular cases of {\lq associated fibre bundles\rq} in \cite{Brion18}.)
        
        \item If $Y$ is a spherical $G_{0}$-variety, then the parabolic induction $G\times^{P} Y$ is a spherical $G$-variety; see for example \cite[\S20.6]{Timashev}.
    \end{enumerate}
\end{rmk}

In the rest of this section, we denote by $G$ a connected reductive group, and by $B$ a Borel subgroup.

\begin{thm}[{\cite{Hirschowitz1984}, \cite[Th\'eor\`eme 1.3.(i)]{Brion1993}, \cite{FMSS}}] \label{thm: reduction to B cycles}
    Assume that $k$ is of characteristic $0$.
    Let $X$ be a complete $B$-variety.
    Any effective cycle on $X$ is rationally equivalent to an effective $B$-invariant cycle.
\end{thm}
\begin{rmk}
    Theorem \ref{thm: reduction to B cycles} is proved in \cite{Hirschowitz1984} and \cite[Th\'eor\`eme 1.3.(i)]{Brion1993} in the case where $X$ is projective.
    The equivariant Chow lemma \cite[Theorem 2.1.3]{Perrin-sanya} implies the desired statement for $X$ complete.
    In fact, a more general version of the statement is given in \cite[Theorem 1]{FMSS}.
\end{rmk}

Now we recall a classification of $B$-curves (i.e., $B$-invariant curves) on spherical varieties, due to Brion \cite{Brion1993, Brion97}.
To this end, we use the notation of \cite[\S5]{Perrin-sanya}.

\begin{defnthm}[{\cite[\S1.1--1.3]{Brion97}, \cite[\S5]{Perrin-sanya}}] \label{thm:types-of-B-curves}
    Assume that $k$ is of characteristic $0$.
    Let $X$ be a complete spherical $G$-variety.
    If $C$ is an irreducible complete $B$-curve on $X$, then $C$ is isomorphic to $\PP^{1}$, and one of the following holds:
        \begin{enumerate}
            \item $B$ acts on $C$ via a nontrivial character. In this case, we say that $C$ is of \emph{type ($\chi$)};
            \item $C$ is contained in a closed $G$-orbit in $X$. In this case, we say that $C$ is of \emph{type (U)};
            \item There exists a closed surface $Y \subset X$ such that $P:= \text{Stab}_{G}(Y)$ is a parabolic subgroup containing $B$, $Y = P \cdot C$, and the natural morphism
            \[
                G \times^{P} Y \rightarrow X, \quad [g,\, y] \mapsto g \cdot y
            \]
            induces a birational morphism $G \times^{P}Y \rightarrow G \cdot C$.
            Furthermore, the natural $P$-action on $Y$ factors through a surjection $P \twoheadrightarrow PGL_{2}$ where the $PGL_{2}$-action on $Y$ is either
            \begin{enumerate}
                \item the diagonal action on $Y = \PP^{1} \times \PP^{1}$. In this case, we say that $C$ is of \emph{type (T)}; or
                \item the adjoint action on $Y = \PP(\mathfrak{sl}_{2}^{\vee})$. In this case, we say that $C$ is of \emph{type (N)}.
            \end{enumerate}
        \end{enumerate}
\end{defnthm}

The following are some subclasses of spherical varieties, which are used in the sequel.

\begin{defn} \label{defn: symm}
Assume that $k$ is not of characteristic $2$.
    \begin{enumerate}
        \item A closed reduced subgroup $H \subset G$ is called a \emph{symmetric subgroup} if there exists a nontrivial algebraic group involution $\theta : G \rightarrow G$ such that $(G^{\theta})^{0} \subset H \subset G^{\theta}$.
        Here, $G^{\theta}$ is the fixed point subgroup of $\theta$, and the superscript $0$ denotes the identity component.
    
    \item A normal $G$-variety $X$ is called a \emph{symmetric $G$-variety} if $X$ contains an open $G$-orbit isomorphic to $G/H$ for a symmetric subgroup $H \subset G$.
    \end{enumerate}
\end{defn}

\begin{rmk} \label{rmk: symm}
As in Definition \ref{defn: symm}, assume that $k$ is not of characteristic $2$.
    \begin{enumerate}
        \item Symmetric varieties are spherical; see \cite[Theorem 26.14]{Timashev}.
        \item Toric varieties are symmetric (take $\theta$ to be $x \mapsto x^{-1}$).
        \item The $PGL_{2}$-varieties $\PP^{1} \times \PP^{1}$ and $\PP(\mathfrak{sl}_{2}^{\vee})$ in Theorem \ref{thm:types-of-B-curves} are symmetric (take $\theta$ to be the conjugation by $\begin{pmatrix} -1 & 0 \\ 0 & 1 \end{pmatrix}$).
    \end{enumerate}
\end{rmk}

\begin{defn} \label{defn: horo}
    \begin{enumerate}
        \item A closed reduced subgroup $H \subset G$ is called a \emph{horospherical subgroup} if $H$ contains the unipotent radical of a Borel subgroup of $G$.
    
        \item A normal $G$-variety $X$ is called a \emph{horospherical $G$-variety} if $X$ contains an open orbit isomorphic to $G/H$ for a horospherical subgroup $H \subset G$.
    \end{enumerate}
\end{defn}

\begin{rmk} \label{rmk: horo}
    \begin{enumerate}
        \item Horospherical varieties are spherical by the Bruhat decomposition.
        \item Rational homogeneous spaces and toric varieties are horospherical.
        \item For a horospherical subgroup $H \subset G$, the normalizer $N_{G}(H)$ of $H$ is a parabolic subgroup.
        Moreover, the quotient $N_{G}(H)/H$ is isomorphic to a torus, and its dimension is called the \emph{rank} of $G/H$.
        See \cite[\S7.1--7.2]{Timashev} for details.
    \end{enumerate}
\end{rmk}

\begin{defn} \label{defn: toroidal}
    Let $X$ be a spherical $G$-variety.
    We say that $X$ is \emph{toroidal} if every $B$-invariant but not $G$-invariant prime divisor on $X$ does not contain a $G$-orbit.
\end{defn}

\begin{rmk} \label{rmk: toroidal}
    \begin{enumerate}
        \item Rational homogeneous spaces and toric varieties are toroidal.
    
    \item For details on toroidal varieties, we refer to \cite[\S29]{Timashev} and \cite{Perrin-sanya}.
    For example, when $k$ is of characteristic $0$, the following are well known:
    \begin{enumerate}
        \item Every homogeneous spherical variety admits an equivariant toroidal completion.
        Moreover, for any spherical $G$-variety $X$, there exists a toroidal spherical $G$-variety $\hat{X}$ equipped with a $G$-equivariant proper birational morphism $\hat{X} \rightarrow X$.
        These facts can be shown by Luna--Vust theory; see \cite[p. 174]{Timashev} for details.
        \item According to the local structure theorem (\cite[Theorem 29.1]{Timashev}, \cite[Theorem 3.4.1]{Perrin-sanya}), toroidal spherical varieties locally look like toric varieties.

        \item A toroidal horospherical $G$-variety $X$ can be written as the parabolic induction from a toric variety; see for example \cite[Examples 1.13(2)]{pasquier-thesis}.
        More precisely, if the open dense orbit in $X$ is isomorphic to $G/H$, then $X \simeq G \times^{N_{G}(H)} Y$ for a $N_{G}(H)/H$-toric variety $Y$ (see Remark \ref{rmk: horo}).
        Observe that $\dim Y$ is equal to the rank of $G/H$.
    \end{enumerate}
    
    \end{enumerate}
\end{rmk}

We close this section with a description of orbits of $B$-curves of type ($\chi$).

\begin{prop} \label{prop: orbit of chi curve}
    Assume that $k$ is of characteristic $0$.
    Let $X$ be a complete spherical $G$-variety, and $C \subset X$ a $B$-curve of type ($\chi$).
    For a point $x \in C$ that is not $B$-fixed and its stabilizer $H:= \text{Stab}_{G}(x)$, the following hold:
    \begin{enumerate}
        \item $H$ is a horospherical subgroup of $G$ such that $\dim N_{G}(H)/H = 1$.
        That is, $G/H$ is a horospherical variety of rank 1.

        \item For the toric variety $\PP^{1}$ under the $k^{\times}$-action, if we regard $\PP^{1}$ as an $N_{G}(H)$-variety via the quotient morphism $N_{G}(H) \rightarrow N_{G}(H)/H \simeq k^{\times}$, then there is a $G$-equivariant birational morphism $G \times^{N_{G}(H)} \PP^{1} \rightarrow G \cdot C$ such that $N_{G}(H) \times ^{N_{G}(H)}\PP^{1} (\simeq \PP^{1})$ is the strict transform of $C$.
    \end{enumerate}
\end{prop}

\begin{proof}
    \begin{enumerate}
        \item Let $T$ be a maximal torus contained in $B$.
        Since $B$ acts on $C(\simeq\PP^{1})$ via a nontrivial character $\chi : B \rightarrow k^{\times}$, $H$ contains the unipotent radical $U \subset B$ and the kernel of $\chi|_{T} : T \rightarrow k^{\times}$.
        It follows that $H$ is a horospherical subgroup with $\dim N_{G}(H)/H \le 1$.
        If it is zero, then $H$ is a parabolic subgroup, which is a contradiction otherwise $G \cdot x$ is a rational homogeneous space containing $C$.
        Indeed, a rational homogeneous space does not contain a $B$-curve of type ($\chi$), since a rational homogeneous space contains a unique $B$-fixed point, while a $B$-curve of type ($\chi$) contains two $B$-fixed points.
        Therefore $\dim N_{G}(H)/H = 1$ and the statement follows.

        \item Since $C$ is $B$-invariant, $G \cdot C$ is a closed subvariety of $X$.
        By Theorem \ref{thm:facts-on-spherical}(2), we may assume that $X = G \cdot C$, that is, $X$ is a horospherical variety whose open $G$-orbit $G/H$ is of rank 1 and $C$ intersects the open $G$-orbit. Consider a toroidal horospherical $G$-variety $\hat{X}$ equipped with a $G$-equivariant proper birational morphism $\hat{X} \rightarrow X$ (Remark \ref{rmk: toroidal}(2.a)).
        Being a toroidal horospherical variety, $\hat{X}$ can be obtained as the parabolic induction $G \times^{N_{G}(H)} \hat{Y}$ for a $k^{\times}$-toric variety $\hat{Y}$ where the $N_{G}(H)$-action on $\hat{Y}$ is given via the quotient morphism $N_{G}(H) \rightarrow N_{G}(H)/H \simeq k^{\times}$ (Remark \ref{rmk: toroidal}(2.c)).
        Moreover, since the only 1-dimensional toric variety is $\PP^{1}$, we have $\hat{Y} \simeq\PP^{1}$.
        Denote by $\hat{C}$ the strict transform of $C$ in $\hat{X}$.
        Since $\hat{C}$ is a $B$-curve of type ($\chi$), and since a rational homogeneous space does not contain such a curve, $\hat{C}$ is contained in the fiber of $\hat{X} \rightarrow G/N_{G}(H)$ over a $B$-fixed point.
        As $G/N_{G}(H)$ contains a unique $B$-fixed point $e \cdot N_{G}(H)$ and $\hat{Y}$ is the fiber over $e \cdot N_{G}(H)$, we have $\hat{C} = \hat{Y}$. \qedhere
    \end{enumerate}
\end{proof}

\subsection{Characteristic \texorpdfstring{$p$}{p} method: Frobenius splitting} \label{subsection: F splitting}

Now we prove that for an ample line bundle on a spherical variety, the restriction map for its global sections to a $B$-curve is surjective.
To do this, we prove a cohomology vanishing result, using a positive characteristic method called Frobenius splitting, introduced by Mehta and Ramanathan \cite{MR85}.
Our main reference on the Frobenius splitting is \cite{BK05}.

\begin{defn}\label{defn:F-split}
    Assume that $k$ is of positive characteristic.
    Let $X$ be a variety over $k$, and $F:X \rightarrow X$ the Frobenius morphism, that is, $F$ is the identity on the underlying topological space, but the associated map $F^{\#}:\Ocal_{X} \rightarrow F_{*}\Ocal_{X}$ is given by the $p$th power map.
    \begin{enumerate}
        \item We say that $X$ is \emph{Frobenius split} if there is an $\Ocal_{X}$-linear map $\sigma: F_{*}\Ocal_{X} \rightarrow \Ocal_{X}$ such that $\sigma \circ F^{\#} = id_{\Ocal_{X}}$, i.e., $F^{\#}$ has a left inverse.
        In this case, we say that $\sigma$ is a \emph{splitting}.

        \item A closed subscheme $Y$ of $X$ is called \emph{compatibly split} in $X$ if there exists a splitting $\sigma : F_{*}\Ocal_{X} \rightarrow \Ocal_{X}$ such that $\sigma(F_{*} I_{Y/X}) = I_{Y/X}$ for the ideal sheaf $I_{Y/X}$ corresponding to $Y \subset X$.
    \end{enumerate}
\end{defn}

In fact, a closed subscheme that is compatibly split is necessarily reduced (\cite[Proposition 1.2.1]{BK05}).

In the rest of this section, we denote by $G$, $T$, and $B$ a connected reductive group, a maximal torus, and a Borel subgroup containing $T$, respectively.
For a $B$-variety $X$, there is a notion of {\lq almost $B$-invariant\rq} splitting, defined as follows:

\begin{defn}[{\cite[Definition 4.1.4, Lemma 4.1.6]{BK05}}]
    Assume that $k$ is of positive characteristic $p$.
    Let $X$ be a $B$-variety, and suppose that $\sigma : F_{*} \Ocal_{X} \rightarrow \Ocal_{X}$ is a splitting.
    We say that $\sigma$ is \emph{$B$-canonical} if the following two conditions are satisfied:
    \begin{enumerate}
        \item $\sigma$ is $T$-invariant; and
        \item Let $e$ be a root vector associated to a simple root and $m \in \ZZ_{\ge p}$.
        For the $m$-th divided power $e^{(m)} \coloneqq e^{m}/m!$, defined as an element of the \emph{distribution algebra} $\text{Dist}(B)$ of $B$ (cf. \cite[\S7.7]{Jantzen}), $\sigma$ is annihilated by $e^{(m)}$ under the natural action of $\text{Dist}(B)$ (cf. \cite[\S7.11]{Jantzen}) induced by the action of the Lie algebra of $B$ on $\text{Hom}_{\Ocal_{X}}(F_{*}\Ocal_{X},\, \Ocal_{X})$.
    \end{enumerate}
\end{defn}

It is not true that every spherical variety admits a $B$-canonical splitting.
In fact, not every spherical variety admits a Frobenius splitting.
However, by \cite[Corollary 3.4]{Perrin14} or by \cite[Theorem 2.2 and \S3]{Tange18}, a $B$-canonical splitting exists on a symmetric variety when $p > 2$.
Combined with \cite[Proposition 4.1.17, Lemma 1.1.8, and Exercise 4.1.E(3)]{BK05} and \cite[Lemma 1.3]{Tange12}, the existence of $B$-canonical splitting on a symmetric variety implies the following:

\begin{lemma} \label{lemma: B can splitting exists}
    Assume that the characteristic of $k$ is larger than $2$.
    Let $G_{0}$ be a connected reductive group and $Y$ a symmetric $G_{0}$-variety.
    If $P$ is a parabolic subgroup of $G$ containing $B$, then for the parabolic induction $X:= G\times^{P}Y$ induced from $Y$ via a surjection $P \twoheadrightarrow G_{0}$, $Y$ is compatibly split in $X$ by a $B$-canonical splitting.
    (Here, we identify $Y$ with $P \times^{P} Y$ via the natural isomorphism $P \times^{P} Y \simeq Y$, $[p,\, y] \mapsto p \cdot y$.)
\end{lemma}

Using this, we prove a cohomology vanishing result when $\mathrm{char}(k) \not= 2$.

\begin{lemma} \label{lemma: cohomology of parabolic induction}
    Assume that $k$ is of characteristic $p \not= 2$, that is, either $p = 0$ or $p > 2$.
    Let $X$ be a normal projective variety, and $Y$ a closed subvariety of $X$.
    Suppose that there exists a variety $\hat{X}$ equipped with a proper birational morphism $\pi : \hat{X} \rightarrow X$ satisfying the following conditions:
    \begin{enumerate}
        \item $\hat{X}$ is the parabolic induction $G\times^{P}\hat{Y}$ induced from a symmetric $G_{0}$-variety $\hat{Y}$ via a surjection $P \twoheadrightarrow G_{0}$ where $G_{0}$ is a connected reductive group and $P$ is a parabolic subgroup of $G$; and
        \item $Y = \pi(\hat{Y})$.
    \end{enumerate}
    For any ample line bundle $A$ on $X$ and the ideal sheaf $I_{Y/X}$ associated to $Y \subset X$, we have $H^{1}(X,\,A \otimes I_{Y/X}) = 0$.
    In particular, the restriction map $H^{0}(X,\, A) \rightarrow H^{0}(Y,\, A|_{Y})$ is surjective.
\end{lemma}

\begin{proof}
    Observe that the morphism $\pi^{\#} : \Ocal_{X} \rightarrow \pi_{*} \Ocal_{\hat{X}}$ is an isomorphism by Zariski's main theorem \cite[Corollary~III.11.4]{Hartshorne}.
    If $p$ is positive, then $Y$ is compatibly split in $X$ by Lemma \ref{lemma: B can splitting exists} and \cite[Lemma 1.1.8]{BK05}, and hence the statement follows from \cite[Theorem 1.2.8]{BK05}.
    Thus we may assume that $p = 0$.
    In this case, the statement follows from the semi-continuity of cohomologies and the technique of reduction to positive characteristic; see for example \cite[\S1.6]{BK05} and \cite[\S31.3]{Timashev}.
    For the sake of completeness, let us briefly explain the idea of the proof.

    First, we extend necessary data over a finitely generated subring $R$ of $k$.
    This is straightforward for $G_{0}$, $G$, $P$, and $G/P$, since $k$ is an algebraically closed field of characteristic $0$ and so every reductive group over $k$ is split.
    Indeed, by \cite[Corollaire~1.3, Expos\'e~XXV and Proposition~1.4, Expos\'e~XXVI]{SGAIII}, there exist reductive group schemes $\underline{G_{0}}$ and $\underline{G}$ and a parabolic subgroup scheme $\underline{P}$ of $\underline{G}$ over $\ZZ$ such that the base changes of the $\ZZ$-schemes $\underline{G_{0}}$, $\underline{G}$, $\underline{P}$, and $\underline{G}/\underline{P}$ via $\text{Spec}(k) \rightarrow \text{Spec}(\ZZ)$ are $G_{0}$, $G$, $P$, and $G/P$, respectively.
    Similarly, the base change of the canonical morphism $\underline{G} \rightarrow \underline{G}/\underline{P}$ over $\ZZ$ recovers the natural projection $G \rightarrow G/P$.
    
    Other objects may not be defined over $\ZZ$.
    Nonetheless, since a variety can be covered by finitely many affine open subsets together with gluing functions, it is defined over a finitely generated subring $R$ of $k$, that is, it can be obtained as a base change via $\text{Spec}(k) \rightarrow \text{Spec}(R)$ from an $R$-scheme.
    Identifying a morphism between varieties with its graph, which is a variety, we see that finitely many varieties and morphisms between them can be defined over some finitely generated subring $R \subset k$.
    Moreover, over the affine open subsets, an ideal sheaf is nothing but an ideal, which is finitely generated, and a line bundle is a free module of rank 1.
    Thus by adding more finitely many generators to $R$ if necessary, we may assume that finitely many ideal sheaves and line bundles are also defined over $R$.
    We apply this process to find a finitely generated subring $R \subset k$ such that the following are defined over $R$:
    \begin{itemize}
        \item The varieties $X$ and $\hat{X}$, and the morphisms $\pi : \hat{X} \rightarrow X$, $\hat{X} \rightarrow G/P$, and $P \rightarrow G_{0}$.
        \item The ideal sheaves $I_{Y/X}$ and $I_{\hat{Y}/\hat{X}}$ (and hence the subvarieties $Y$ and $\hat{Y}$), and the line bundle $A$.
        \item The $G$-action on $\hat{X}$, and the $G_{0}$-action on $\hat{Y}$, by identifying them with the morphisms $G \times \hat{X} \rightarrow \hat{X}$ and $G_{0} \times \hat{Y} \rightarrow \hat{Y}$.
    \end{itemize}
    Their extensions over $R$ are denoted by the same symbols with underlined.
    For example, the base change of the morphism $\underline{\pi} : \underline{\hat{X}} \rightarrow \underline{X}$ via $\text{Spec}(k) \rightarrow \text{Spec}(R)$ is $\pi : \hat{X} \rightarrow X$.
    We denote by $r \in \text{Spec}(R)$ a closed point, and by $\overline{r}$ the associated geometric point $\text{Spec}(\overline{\kappa(r)})$ where $\overline{\kappa(r)}$ is the algebraic closure of the residue field $\kappa(r)$ of $r$.
    The fibers over $r$ and $\overline{r}$ are denoted by the same symbols without underlined.
    For example, a geometric fiber of $\underline{\pi} : \underline{\hat{X}} \rightarrow \underline{X}$ is denoted by $\pi_{\overline{r}} : \hat{X}_{\overline{r}} \rightarrow X_{\overline{r}}$.
    At this point, we do not claim that the extensions over $R$ and their geometric fibers have the exactly same properties with the original objects over $k$.
    Namely, it is not clear whether $X_{\overline{r}}$ is normal and projective, $A_{\overline{r}}$ is ample, $\pi_{\overline{r}}$ is proper and birational, $\hat{Y}_{\overline{r}}$ is $(G_{0})_{\overline{r}}$-symmetric, and $\hat{X}_{\overline{r}}$ is the parabolic induction $G_{\overline{r}} \times^{P_{\overline{r}}} \hat{Y}_{\overline{r}}$.
    These properties will be obtained in the next two paragraphs, again by adding finitely many more generators to $R$, or equivalently, by taking localizations of $R$ in $k$.
    
    The point is that for a general closed point $r \in \text{Spec}(R)$, the geometric fibers over $\overline{r}$ inherit many properties from the given objects defined over $k$.
    For example, by \cite[Th\'eor\`eme (8.10.5), p. 37--38]{EGAIV3}, we may assume that $\underline{X}$ is a projective $R$-scheme by replacing $\text{Spec}(R)$ by its affine open subset.
    Since geometric normality and ampleness are Zariski open conditions in a proper and flat family by \cite[Th\'eor\`eme~(12.2.4), p.~183 and Corollaire~(9.6.4), p.~75]{EGAIV3}, we may assume that the geometric fiber $X_{\overline{r}}$ of $\underline{X}$ is normal and that $\underline{A}$ is ample over $R$.
    We can further assume that $\underline{\pi}$ is a proper birational morphism over $R$, and that $\underline{P} \rightarrow \underline{G_{0}}$ is surjective, again by \cite[Th\'eor\`eme (8.10.5), p. 37--38]{EGAIV3}.

    On the other hand, the residue fields of closed points of $\text{Spec}(R)$ are of positive characteristic, and their prime fields are the images of the subring $\ZZ \subset R$.
    Thus by adding $\frac{1}{2}$ to the generating set of $R$, we may assume that the residue fields are not of characteristic $2$.
    Considering an affine open subset of $\text{Spec}(R)$ where the involution $\theta: G_{0} \rightarrow G_{0}$ extends, we may assume that the geometric fiber $\hat{Y}_{\overline{r}}$ of $\underline{\hat{Y}}$ is a symmetric $(G_{0})_{\overline{r}}$-variety.
    Since $\hat{Y}_{\overline{r}}$ is the fiber of the morphism $\hat{X}_{\overline{r}} \rightarrow G_{\overline{r}} / P_{\overline{r}}$ at $e \cdot P_{\overline{r}}$, we have a $G_{\overline{r}}$-equivariant bijective morphism $G_{\overline{r}} \times^{P_{\overline{r}}} \hat{Y}_{\overline{r}} \rightarrow \hat{X}_{\overline{r}}$, $[g,\,y]\mapsto g\cdot y$.
    Since $\hat{X} = G\times^{P}\hat{Y}$, we can assume that $G_{\overline{r}} \times^{P_{\overline{r}}} \hat{Y}_{\overline{r}}$ and $\hat{X}_{\overline{r}}$ contain open dense $G$-orbits isomorphic to each other (more precisely, isomorphic to $G_{\overline{r}}/H_{\overline{r}}$ where an open dense orbit in $\hat{X}$ is isomorphic to $G/H$), and then the morphism $G_{\overline{r}} \times^{P_{\overline{r}}} \hat{Y}_{\overline{r}} \rightarrow \hat{X}_{\overline{r}}$ is birational.
    Being a bijective birational morphism between normal complete varieties, it is an isomorphism by Grothendieck's version of Zariski's main theorem \cite[Corollaire~(8.12.10), p.~48]{EGAIV3}.
    That is, $\hat{X}_{\overline{r}}$ is the parabolic induction $G_{\overline{r}} \times^{P_{\overline{r}}} \hat{Y}_{\overline{r}}$.
    Since $\pi_{\overline{r}} : \hat{X}_{\overline{r}} \rightarrow X_{\overline{r}}$ sends $\hat{Y}_{\overline{r}}$ onto $Y_{\overline{r}}$, by the case of positive characteristic, we have $H^{1}(X_{\overline{r}},\, A_{\overline{r}} \otimes I_{Y_{\overline{r}}/X_{\overline{r}}}) = 0$.
    Since the restriction of $\underline{A} \otimes_{R} I_{\underline{Y} / \underline{X}}$ on the geometric fiber $X_{\overline{r}}$ is $A_{\overline{r}} \otimes I_{Y_{r}/X_{\overline{r}}}$,
    we see that $H^{1}(X_{r},\, (\underline{A} \otimes_{R} I_{\underline{Y} / \underline{X}})_{r}) = 0$ by the cohomology base change theorem \cite[Proposition~III.9.3]{Hartshorne} applied to the base change $\text{Spec} (\overline{\kappa(r)}) \rightarrow \text{Spec}(\kappa(r))$.
    Thus by the semi-continuity \cite[Theorem~III.12.8]{Hartshorne}, we have $H^{1}(X_{\eta},\, (\underline{A} \otimes_{R} I_{\underline{Y} / \underline{X}})_{\eta}) = 0$ for the generic point $\eta \in \text{Spec}(R)$.
    Finally, again by the cohomology base change theorem \cite[Proposition~III.9.3]{Hartshorne} applied to the base change $\text{Spec}(k) \rightarrow \text{Spec}(\kappa(\eta))$, we conclude that $H^{1}(X,\, A \otimes I_{Y/X}) = 0$.
\end{proof}

Combining Lemma~\ref{lemma: cohomology of parabolic induction} with the results on $B$-curves given in Section~\ref{subsection:pre}, we prove Theorem~\ref{thm: res for ample is surjective}:

\begin{proof}[Proof of Theorem~\ref{thm: res for ample is surjective}]
    Since $C$ is $B$-invariant, $G \cdot C$ is a closed subvariety of $X$.
    By Theorem \ref{thm:facts-on-spherical}(2)--(3), $G \cdot C$ is a (projective) spherical $G$-variety, and the restriction map $H^{0}(X,\, A) \rightarrow H^{0}(G \cdot C,\, A|_{G \cdot C})$ is surjective.
    Thus by replacing $(X,\, A)$ by $(G \cdot C,\, A|_{G \cdot C})$, we may assume that $C$ intersects the open $G$-orbit in $X$.
    By Theorem \ref{thm:types-of-B-curves}, there are three possible cases:
    \begin{itemize}
        \item $C$ is of type (U).
        In this case, since $X$ is a rational homogeneous space and $C$ is a Schubert curve, the restriction map $H^{0}(X,\, A) \rightarrow H^{0}(C,\, A|_{C})$ is surjective by \cite[Theorem 3]{MR85}.

        \item $C$ is of type ($\chi$).
        In this case, the statement follows from Proposition \ref{prop: orbit of chi curve}(2) and Lemma \ref{lemma: cohomology of parabolic induction}.
        
        \item $C$ is of type (T) or (N).
        In this case, by Theorem \ref{thm:types-of-B-curves}(3), there exists a closed subvariety $Y \subset X$ satisfying the following conditions: $Y$ is a symmetric $PGL_{2}$-surface (Remark \ref{rmk: symm}(3)), $P := \text{Stab}_{G}(Y)$ is a parabolic subgroup acting on $Y$ via a surjection $P \twoheadrightarrow PGL_{2}$, $Y = P \cdot C$, and the morphism $G \times^{P} Y \rightarrow X$, $[g,\,y]\mapsto g \cdot y$ is birational.
        The restriction map $H^{0}(X,\, A) \rightarrow H^{0}(Y,\, A|_{Y})$ is surjective by Lemma \ref{lemma: cohomology of parabolic induction}, and thus it suffices to show that the restriction map $H^{0}(Y,\, A|_{Y}) \rightarrow H^{0}(C,\, A|_{C})$ is surjective.
        To see this, observe that the surjection $P \twoheadrightarrow PGL_{2}$ sends $B$ to a Borel subgroup $B_{0} \subset PGL_{2}$, and so $C$ is a $B_{0}$-invariant but not $PGL_{2}$-invariant curve on $Y$.
        Using the explicit description of $Y$ given in Theorem \ref{thm:types-of-B-curves}(3), such a curve can be easily classified.
        For this purpose, we may assume that $B_{0}$ consists of upper triangular matrices.
        \begin{itemize}
            \item If $C$ is of type (T), then $Y \simeq \PP^{1} \times \PP^{1}$ and $PGL_{2}$ acts diagonally.
            The only irreducible complete $B_{0}$-invariant but not $PGL_{2}$-invariant curves on $Y$ are $\PP^{1} \times \infty$ and $\infty \times \PP^{1}$ for a (unique) $B_{0}$-fixed point $\infty \in \PP^{1}$.
            Thus $H^{0}(Y,\, A|_{Y}) \rightarrow H^{0}(C,\, A|_{C})$ is surjective by the K\"unneth formula.

            \item If $C$ is of type (N), then $Y \simeq \PP(\mathfrak{sl}_{2}^{\vee})$ and $PGL_{2}$ acts via the adjoint representation.
            That is, if we identify $\mathfrak{sl}_{2}$ with the set of traceless $2 \times 2$ matrices, $PGL_{2}$ acts via the conjugation.
            Under this identification, there are exactly four $B_{0}$-orbits in $Y$:
            \begin{align*}
                \Ocal_{1} &\coloneqq \left\{ [z] \in \PP(\mathfrak{sl}_{2}^{\vee}) : z = \begin{pmatrix} a & b \\ c & -a \end{pmatrix},\, a,\, b \in k, \, c\in k^{\times},\, \det(z) \not=0 \right\}, \\
                \Ocal_{2} &\coloneqq \left\{ [z] \in \PP(\mathfrak{sl}_{2}^{\vee}) : z = \begin{pmatrix} a & b \\ 0 & -a \end{pmatrix},\, a,\, b \in k,\, \det(z)\not=0 \right\}, \\
                \Ocal_{3} &\coloneqq \left\{ [z] \in \PP(\mathfrak{sl}_{2}^{\vee}) : z = \begin{pmatrix} a & b \\ c & -a \end{pmatrix},\, a,\, b \in k,\, c\in k^{\times},\, \det(z) = 0 \right\}, \\
                \Ocal_{4} &\coloneqq \left\{ \left[\begin{pmatrix} 0 & 1 \\ 0 & 0 \end{pmatrix} \right] \right\}.
            \end{align*}
            Observe that $\Ocal_{1}$ is open in $\PP(\mathfrak{sl}_{2}^{\vee})$.
            In particular, there are only two irreducible complete $B_{0}$-curves in $Y$, namely $\Ocal_{2} \cup \Ocal_{4}$ and $\Ocal_{3} \cup \Ocal_{4}$.
            Since $\Ocal_{3} \cup \Ocal_{4}$, the vanishing locus of $\det(z)$, is $PGL_{2}$-invariant but $C$ is not $PGL_{2}$-invariant, we have $C = \Ocal_{2} \cup \Ocal_{4}$, i.e.,
            \[
                C = \left\{ [z] \in \PP(\mathfrak{sl}_{2}^{\vee}) : 0\not= z = \begin{pmatrix} a & b \\ 0 & -a \end{pmatrix},\, a,\, b \in k \right\}.
            \]
            Therefore under the isomorphism $\PP(\mathfrak{sl}_{2}^{\vee}) \simeq \PP^{2}$ induced by
            \[
                \mathfrak{sl}_{2} \simeq \CC^{3}, \quad \begin{pmatrix} a & b \\ c & -a \end{pmatrix} \mapsto (a,\,b,\,c),
            \]
            $C$ is identified with a line on $\PP^{2}$, and hence $H^{0}(Y,\, A|_{Y}) \rightarrow H^{0}(C,\, A|_{C})$ is surjective. \qedhere
        \end{itemize}
    \end{itemize}
\end{proof}

\subsection{Positivity of twisted kernel bundles} \label{subsection: Nakai}
From now on, we always assume that the characteristic of the base field $k$ is zero.
In this section, we investigate positivity of the twisted kernel bundles associated to ample line bundles on spherical varieties.

\begin{defn} \label{defn:kernel}
    Let $L$ be a globally generated vector bundle on a complete variety $X$.
    The \textit{kernel bundle} $M_L$ is defined to be the kernel of the evaluation map $\text{ev}: H^{0}(X,\, L) \otimes \Ocal_{X} \rightarrow L$.
    That is, $M_{L}$ is a vector bundle on $X$ that fits into a short exact sequence
    \[
        \begin{tikzcd}
        0 \arrow{r} & M_L \arrow{r} & H^0(X,\,L) \otimes \Ocal_X \arrow[r, "\text{ev}"] & L \arrow{r} & 0.
    \end{tikzcd}
    \]
\end{defn}

As before, we denote by $G$ a connected reductive group, and by $B$ a Borel subgroup.

Recall that nef line bundles on a complete spherical variety are globally generated (Theorem \ref{thm:facts-on-spherical}(3)).
In particular, for an ample line bundle on a spherical variety, its kernel bundle is well defined.
Now the main result of this section can be stated as follows:

\begin{lemma} \label{lem:positivity-of-kernel-bundle}
    Let $X$ be a projective spherical $G$-variety.
    If $L$ and $A$ are ample line bundles on $X$, then the twisted kernel bundle $M_L \otimes A$ is nef on $X$.
\end{lemma}

To prove Lemma \ref{lem:positivity-of-kernel-bundle}, we generalize a toric Nakai criterion \cite[Theorem 2.1]{HMP2010} to the setting of spherical varieties.

\begin{proof}[Proof of Theorem~\ref{thm:spherical-Nakai-criterion}]
    \begin{enumerate}
        \item Assume that $E$ is nef on $B$-curves.
        In other words, for every $B$-curve $C' \subset X$, $\Ocal_{\PP(E|_{C'})}(1) \simeq \Ocal_{\PP(E)}(1)|_{\PP(E|_{C'})}$ is nef.
        Since $E$ is $B$-equivariant, $\PP(E)$ admits a $B$-action such that the projection $\pi: \PP(E) \rightarrow X$ is $B$-equivariant, and so $\Ocal_{\PP(E)}(1)$ is nef on $B$-curves on $\PP(E)$.
        By Theorem~\ref{thm: reduction to B cycles}, for any curve $C \subset \PP(E)$, the $1$-cycle $[C]$ is rationally equivalent to $\sum_{\nu=1}^{\ell} a_{\nu} [C_{\nu}]$ for some $a_{\nu} \in \mathbb{Z}_{>0}$ and $B$-curves $C_{\nu} \subset \PP(E)$.
        It follows that $\Ocal_{\PP(E)}(1)$ is nef on $C$, and hence on $\PP(E)$.

        \item The second statement follows from the first statement, finiteness of the number of $B$-curves (Theorem \ref{thm:facts-on-spherical}(1)) and \cite[Remark 2.2]{HMP2010}.
    \end{enumerate}
\end{proof}

\begin{proof}[Proof of Lemma~\ref{lem:positivity-of-kernel-bundle}]
By \cite[Corollary C.5]{Timashev}, we may assume that $L$ and $A$ are $B$-equivariant so that $\text{ev}: H^{0}(X,\, L) \otimes \Ocal_{X} \rightarrow L$ is $B$-equivariant (here we use the natural $B$-action on $\Ocal_{X}$ and the induced $B$-representation structure on $H^{0}(X,\,L)$), and hence the twisted kernel bundle $M_L \otimes A$ is $B$-equivariant.
By Theorem \ref{thm:spherical-Nakai-criterion}(1), it suffices to show that the restriction of $M_{L} \otimes A$ on a $B$-curve $C \subset X$ is nef.
To do this, consider the following commutative diagram:
\begin{equation} \label{diagram:kernel bundle on B curves}
\begin{tikzcd}
    & 0 \arrow{d} & 0 \arrow{d} &  & \\
    & H^0(X,\, L \otimes I_{C/X}) \otimes \Ocal_C \arrow[r, equal] \arrow{d} & H^0(X,\, L \otimes I_{C/X}) \otimes \Ocal_C \arrow{d} & & \\
    0 \arrow{r} & M_{L}|_C \arrow{r} \arrow{d} & H^0(X,\, L) \otimes \Ocal_C \arrow{r} \arrow{d} & L|_C \arrow{r} \arrow[d, equal] & 0 \\
    0 \arrow{r} & M_{L|_C} \arrow{r} & H^0(C,\, L|_C) \otimes \Ocal_C \arrow{r} & L|_C \arrow{r} & 0.
\end{tikzcd}
\end{equation}
Here, $I_{C/X}$ is the ideal sheaf of $C$ in $X$.
Since $C \simeq \PP^{1}$, $h^{i}(C,\, M_{L|_{C}}) = 0$ for all $i \ge 0$ and hence $M_{L|_C}$ is isomorphic to $\Ocal_{\PP^1}(-1)^{\oplus e}$ for $e := h^0(C,\,L|_C)-1$.
By Theorem~\ref{thm: res for ample is surjective}, the restriction map $H^0(X,\, L) \rightarrow H^0(C,\, L|_C)$ is surjective, and thus by the snake lemma, $M_{L}|_{C} \rightarrow M_{L|_{C}}$ is surjective.
In other words, the first column of the diagram (\ref{diagram:kernel bundle on B curves}) is a short exact sequence.
Since $M_{L|_{C}} \simeq \Ocal_{\PP^1}(-1)^{\oplus e}$, the first column is split, i.e., $M_{L}|_{C}$ is a direct sum of copies of $\Ocal_{\PP^1}$ and $\Ocal_{\PP^1}(-1)$.
Since $A$ is ample, $M_L \otimes A$ is nef on $C$, as desired.
\end{proof}

Let us record a corollary of the proof of Lemma \ref{lem:positivity-of-kernel-bundle}, which is analogous to \cite[Corollary~4.8]{HMP2010}:

\begin{coro}
    Let $X$ be a projective spherical $G$-variety.
    If $L$ is an ample line bundle and $C \subset X$ an irreducible complete $B$-curve, then $M_L|_C \simeq \Ocal_{\PP^1}^{\oplus d} \oplus \Ocal_{\PP^1}(-1)^{\oplus e}$ for $d := h^0(X,\,L \otimes I_{C/X})$ and $e := h^0(C,\,L|_C)-1$.
\end{coro}

\section{Hyperbolicity of adjoint bundles on spherical varieties} \label{sec:hyperbolicity-conjecture}

This whole section is devoted to the proof of our main theorem, Theorem \ref{main thm: hyperbolic if orbit closures are smooth}.
The key point is that the method of Coskun and Riedl \cite{CR2023} can be applied to our situation by the nefness of the twisted kernel bundles obtained in Lemma \ref{lem:positivity-of-kernel-bundle}.

In this section, we assume that $k$ is an algebraically closed field of characteristic 0 that is uncountable.

\subsection{Setup} \label{subsection: coskun-riedl}

First, we recall the technique of Coskun and Riedl \cite{CR2023} for proving algebraic hyperbolicity of very general hypersurfaces on almost homogeneous varieties (see also \cite{Ein1988, CR2004, MY2024}).

From now on, for a smooth irreducible complete curve $C$, we denote by $g(C)$ its genus.

\begin{defn} \label{defn: hyp}
    Let $S$ be a smooth projective variety.
    \begin{enumerate}
        \item We say that $S$ is \textit{algebraically hyperbolic} if there is an ample line bundle $A$ and a constant $\epsilon > 0$ such that for any generically injective morphism $f: C \rightarrow S$ from a smooth irreducible complete curve $C$, we have
    \begin{equation} \label{eqn:algebraic-hyperbolicity-2}
        2 g(C)-2 \geq \epsilon  \deg_C(f^{*}A).
    \end{equation}

    \item For a closed subvariety $Z \subsetneq S$, we say that $S$ is \emph{algebraically hyperbolic modulo $Z$} if there is an ample line bundle $A$ and a constant $\epsilon > 0$ such that the inequality (\ref{eqn:algebraic-hyperbolicity-2}) is satisfied by all $f: C \rightarrow S$ with $f(C) \not\subset Z$.
    \end{enumerate}
    
\end{defn}

\begin{notation} \label{notation:almost-hom}
Let $G$ be a linear algebraic group and $X$ a smooth projective $G$-variety.
Suppose that $X$ contains an open dense $G$-orbit $X^{\circ}$.
Let $A$ and $L$ be $G$-equivariant, globally generated and ample line bundles.
Let $S$ be a very general element of the linear system $|L|$, and without loss of generality, we assume that $S$ is smooth and irreducible.
Let $f:C \rightarrow S$ be a generically injective map from a smooth irreducible complete curve $C$.
\end{notation}

\begin{construction} \label{construction:hyp}
In Notation~\ref{notation:almost-hom}, let $\Scal_1 \rightarrow \Vcal \subset H^0(X,\,L)$ be the universal family of the zero loci of general sections of $L$.
By \cite[\S2.1]{CR2023}, we have a $G$-equivariant \'{e}tale morphism $\Ucal \rightarrow \Vcal$, a smooth family $\Ycal \rightarrow \Ucal$ of curves of genus $g(C)$, and a generically injective morphism $i : \Ycal \rightarrow \Scal := \Scal_{1} \times_{\Vcal} \Ucal$ whose fibers (over points in $\Ucal$) are deformations of $f : C \rightarrow S$.
We denote by $\pi_{1} : \Scal \rightarrow \Ucal$ and $\pi_{2} : \Scal \rightarrow X$ the natural projections.
For an element $t \in \Ucal$, let $C_t$ and $S_t$ be the fibers of $\Ycal$ and $\Scal$, respectively, and let $i_t:C_t \rightarrow S_t$ be the restriction of $i$ to $C_t$.
Define the normal sheaves $N_{i/\Scal}$ and $N_{i_{t}/S_{t}}$ as the cokernels of the injective maps $T_{\Ycal} \hookrightarrow i^{*}T_{\Scal}$ and $T_{C_{t}} \hookrightarrow i_{t}^{*} T_{S_{t}}$, respectively.
Similarly, define the vertical tangent sheaf $T_{\Scal/X}$ as the kernel of the surjection $T_{\Scal} \twoheadrightarrow \pi_{2}^{*}T_{X}$.
\end{construction}

\begin{lemma}[{\cite[Proposition~2.1]{CR2023}}] \label{lem:variational-method}
    In Notation~\ref{notation:almost-hom} and Construction~\ref{construction:hyp}, let $t \in \Ucal$ be a general element.
    \begin{enumerate}
        \item $T_{\Scal/X} \simeq \pi_2^{\ast} M_{L}$.
        \item $N_{i_t/S_t} \simeq N_{i/S}|_{C_t}$.
        \item If $f(C)$ meets $X^{\circ}$, then there exists a generic surjection $i_t^{\ast}T_{\Scal/X} \rightarrow N_{i_t/S_t}$, that is, a morphism between $\Ocal_{C_{t}}$-modules whose cokernel is a torsion sheaf.
    \end{enumerate}
\end{lemma}

\begin{rmk}
    In the proof of our main theorems, we shall consider hypersurfaces $S$ satisfying the properties in Lemma \ref{lem:variational-method} for all $f: C \rightarrow S$ meeting $X^{\circ}$.
    Note that $\Ucal$, parametrizing deformations of $f$, is depending on the genus of $C$ and the degree of $f(C)$, and hence our argument applies to very general hypersurfaces.
\end{rmk}

\subsection{Proof of Main Theorem \ref{main thm: hyperbolic if orbit closures are smooth}} \label{sec: Proof}

From now on, we use the following notation, which is also consistent with Notation~\ref{notation:almost-hom} and Construction~\ref{construction:hyp}:
\begin{notation} \label{notation:spherical}
Let $G$ be a connected reductive group.
We denote by $X$ a smooth projective spherical $G$-variety of dimension $n \ge 2$.
Let $X^{\circ}$ be the open dense $G$-orbit in $X$, and put $\partial X := X \setminus X^{\circ}$.
Let $N$ be a nef line bundle on $X$, and $A$ an ample line bundle on $X$.
Put $L := N \otimes A^{\otimes m}$ for $m \geq 2n$, and let $S \in |L|$ be a very general hypersurface.
By replacing $G$ by a finite cover, we may assume that $G$ is the product of a torus and a simply connected semi-simple Lie group so that $N$, $A$, and $L$ are $G$-equivariant.
Since $L$ is ample (and hence globally generated by Theorem \ref{thm:facts-on-spherical}(3)), a very general hypersurface $S \in |L|$ is smooth and irreducible.
\end{notation}

\begin{thm}\label{main thm: pseudo hyperbolic}
    In Notation~\ref{notation:spherical}, for a generically injective morphism $f : C \rightarrow S$ from a smooth irreducible complete curve $C$ to a very general element $S \in |L|$ such that $f(C) \cap X^{\circ} \not= \emptyset$, we have an inequality
    \begin{equation} \label{eqn:genus bound}
        2g(C) - 2 \ge (m-2n+1) \deg_{C}f^{*}A.
    \end{equation}
    In particular, if $m \ge 2n$, then $S$ is algebraically hyperbolic modulo $\partial X$.
\end{thm}

\begin{proof}
By the definition of the normal sheaf $N_{f/S}$,
\begin{equation} \label{eqn:genus and normal sheaf}
    2g(C)-2 = \deg_{C} \omega_C = \deg_{C} f^{*}\omega_S + \deg_{C} N_{f/S}.
\end{equation}
The adjunction formula and the fact that $\omega_X \otimes A^{\otimes (n+1)}$ is nef (\cite[Example 1.5.35]{LazarsfeldI}) imply that
\begin{equation} \label{eqn:bound-for-degree-of-canonical-sheaf}
    \deg_{C} f^{*}\omega_S = \deg_{C} f^{*}\omega_X + \deg_{C} f^{*}L \geq (m-n-1) \deg_{C} f^{*}A.
\end{equation}
Since there is a generic surjection $f^{*} M_{L} \rightarrow N_{f/S}$ (Lemma \ref{lem:variational-method}) and $M_{L} \otimes A$ is nef (Lemma \ref{lem:positivity-of-kernel-bundle}), we have
\begin{equation} \label{eqn: bound for normal sheaf}
    \deg_{C} N_{f/S} = \deg_{C} (N_{f/S} \otimes f^{\ast}A) - (n-2)\deg_{C} f^{*}A \geq -(n-2) \deg_{C} f^{*}A.
\end{equation}
Combining (\ref{eqn:genus and normal sheaf}--\ref{eqn: bound for normal sheaf}), we deduce the desired inequality (\ref{eqn:genus bound}).
\end{proof}

As a corollary, Conjecture \ref{conj: adjoint bundles} holds for $X$ modulo $\partial X$, in the following sense:

\begin{coro} \label{coro: conj modulo boundary}
    In Notation~\ref{notation:spherical}, if $m \ge 3n+1$, then very general elements of $|\omega_{X} \otimes A^{\otimes m}|$ are algebraically hyperbolic modulo $\partial X$.
\end{coro}

\begin{proof}
    The statement follows from Theorem \ref{main thm: pseudo hyperbolic} and the equation (\ref{eqn:decomp adjoint}) in Introduction.
\end{proof}

\begin{proof} [Proof of Theorem~\ref{main thm: hyperbolic if orbit closures are smooth}]
    Now assume that every $G$-orbit closure in $X$ is smooth.
    Let $S \in |L|$ be a very general hypersurface, and $f:C \rightarrow S$ a generically injective morphism from a smooth irreducible complete curve $C$.
    Let $Y$ be the smallest $G$-invariant closed subvariety containing $f(C)$.
    By Theorem \ref{thm:facts-on-spherical}(2), $Y$ is a projective spherical $G$-variety, which is smooth by the assumption.
    Moreover, the restriction map $H^0(X,\,L) \rightarrow H^0(Y,\,L|_Y)$ is surjective, and thus a very general element in $|L|$ restricts to a very general element in $|L|_Y|$.
    Since $C$ intersects the open $G$-orbit in $Y$, by Theorem~\ref{main thm: pseudo hyperbolic}, we have
    \[
        2g(C)-2 \geq (m-2\dim Y+1)\deg_{C} f^{*}A \ge (m-2n+1)\deg_{C} f^{*}A,
    \]
    and it follows that $S$ is algebraically hyperbolic.
\end{proof}

\end{document}